\documentclass{psapm-l}
\usepackage{color}
\usepackage{graphicx}

\newtheorem{theorem}{Theorem}[section]
\newtheorem{proposition}{Proposition}[section]

\theoremstyle{definition}
\newtheorem{definition}[theorem]{Definition}

\theoremstyle{remark}

\numberwithin{equation}{section}

\begin{document}

\title[Two-Fluid Model of Electron-Positron Plasma]
{Collisionless Magnetic Reconnection in a Five-Moment
Two-Fluid Electron-Positron Plasma}

\author{E.A. Johnson}
\address{Department of Mathematics, University of Wisconsin, Madison WI 53717}
\email{ejohnson@math.wisc.edu}

\author{J.A. Rossmanith}
\address{Department of Mathematics, University of Wisconsin, Madison WI 53717}
\email{rossmani@math.wisc.edu}
\thanks{This work is supported in part by NSF Grant DMS-0711885
and a Graduate Fellowship from the Wisconsin Space Grant Consortium.}

\date{November 30, 2008.}



\begin{abstract}
We simulate magnetic reconnection in electron-positron (pair)
plasma using a collisionless two-fluid model with isotropic
pressure. In this model the resistive, Hall, and electrokinetic
pressure terms are absent from the curl of Ohm's law, leaving
the inertial term alone to provide for magnetic reconnection.
Our simulations suggest that for pair plasma simulated with
isotropic pressure fast reconnection does not occur without the
aid of sufficient (numerical) diffusion. We contrast this result
with simulations and published results showing fast reconnection
for collisionless two-fluid plasma with isotropic pressures
and \emph{non}-canceling mass-to-charge ratios, where Hall
effects are present and numerical diffusion is small, and with
published PIC studies of pair plasma which observe fast reconnection
and attribute it to nonisotropic pressure.

\end{abstract}

\maketitle

\section{Introduction}
An important issue of controversy in the magnetic reconnection
community is the minimal conditions required for
fast magnetic reconnection to occur in a plasma
and the minimal modeling requirements to resolve it  \cite{book:PrFo07}.
The first attempt to model reconnection was carried out
by Sweet \cite{article:Sweet58} and Parker \cite{article:Parker57}, who
used a resistive magnetohydrodynamic (MHD) model. 
Their approach was only successful in modeling a slow form of reconnection
and not the much faster reconnection that is observed in laboratory
and space plasma.

The Geospace Environmental Modeling (GEM) Reconnection Challenge
problem was introduced in \cite{article:GEM} and studied
with a variety of models to identify the essential physics
required to model collisionless magnetic reconnection. This
original GEM article concluded that all models that include
the Hall term in the generalized Ohm's law produced essentially
indistinguishable rates of reconnection.  The only other model
in their study that admitted fast reconnection was MHD with large anomalous
(e.g., current-dependent) resistivity, although as expected
it did not exhibit the quadrupole out-of-plane
magnetic field pattern that appears to characterize
models which incorporate the Hall term.

Since inclusion of Hall effects had been identified as the
critical ingredient to admit fast reconnection, Bessho and
Bhattacharjee
studied electron-positron
plasmas, for which the Hall term is
zero\cite{article:Bessho05,article:Bessho07}.
Their particle-in-cell (PIC) simulations
of the GEM problem (with the mass ratio reset from 25 to 1)
exhibited fast reconnection for temperature ratios of 5 and 1.
They found that the structure of the reconnection region still
shows an X-point, although as expected the out-of-plane magnetic
field does not show the quadrupole structure that appears in
models that incorporate Hall term effects.

%
%

This prompted us to ask if it is possible to get fast reconnection
if the Hall term is absent and the pressures are modeled as
isotropic.  Therefore we chose to study reconnection in
an electron-positron plasma using a two-fluid model with
isotropic pressure.


\section{Physical models}

\subsection{Particle-in-cell description.}
Plasma consists of charged particles interacting with the electromagnetic field.
In the absence of gravity and quantum-mechanical effects the
particles of a plasma satisfy Maxwell's equations and 
the Lorentz force to govern particle motion:
\def\B{\mathbf{B}}
\def\E{\mathbf{E}}
\def\J{\mathbf{J}}
\def\v{\mathbf{v}}
\def\x{\mathbf{x}}
\def\Div{\nabla\cdot}
\def\curl{\nabla\times}
\begin{align*}
   &\partial_t \B = -\curl \E,                 & \Div\B&=0,
\\ &\partial_t \E = c^2\curl B - \J/\epsilon_0, & \Div\E&=\sigma/\epsilon_0,
\\ &d_t \tilde\v_p = \frac{q_p}{m_p}
      \Big(\E(\x_p)+\v_p\times\B(\x_p)\Big),
   &d_t \x_p &= \v_p,
\\ &\J = \sum_p S_p(\x_p) q_p \v_p, 
   &\sigma &= \sum_p S_p(\x_p) q_p;
\end{align*}
here $\B$ is magnetic field,
$\E$ is electric field,
$c$ is the speed of light,
$\epsilon_0$ is electric permittivity,
$p$ is particle index,
$\x_p(t)$ is particle position,
$\tilde\v_p(t)=\gamma_p\v_p$ is (proper) particle velocity,
where $\gamma=\big(1-(v/c)^2\big)^{-1/2}\approx 1$ is the Lorentz factor,
$q_p$ is particle charge,
$m_p$ is particle mass,
$\sigma$ is charge density,
$\J$ is current density, and
$S_p(\x-\x_p)$ is particle charge distribution
(e.g., a unit impulse function).

Particle-in-cell (PIC) codes model plasma by attempting to evolve
in a mixed Eulerian-Lagrangian framework: particles are treated
in a Lagrangian way, while the electromagnetic field sits on an
Eulerian computational mesh.

\subsection{Boltzmann description.}
The Boltzmann model replaces the particles with particle (probability)
density functions $f_s(\x,\tilde\v,t)$
of space, velocity, and time, for each species $s$.
The Boltzmann equation asserts conservation (or balance) of particles
in phase space:
\begin{gather*}
  \partial_t f_s + \nabla_\x\cdot(\v f_s)
  + \frac{1}{r_g}\nabla_{\tilde\v}\cdot\Big(\frac{q_s}{m_s}(\E+\v\times\B)f_s\Big)=C_s;
\end{gather*}
here
$\tilde\v=\gamma\v\approx\v$ is (proper) velocity in phase space,
and $C_s$ is a collision
operator which is a function of $\{\tilde \v \mapsto f_p(t,\x,\tilde \v)\}_p$,
where $p$ ranges over all species.
The collisionless Boltzmann equation (alias {\em Vlasov equation})
asserts that $C_s=0$.
Maxwell's equations are coupled to the Boltzmann equation by
the relations
\begin{align*}
   \J &= \sum_s \int_\v f_s q_s \v, 
   &\sigma &= \sum_s \int_\v f_s q_s.
\end{align*}
\def\curl{\nabla\times}
\def\Div{\nabla\cdot}
\def\qdens{\sigma}
\def\E{\mathbf{E}}
\def\B{\mathbf{B}}
\def\J{\mathbf{J}}
\def\R{\mathbf{R}}
\def\q{\mathbf{q}}
\def\u{\mathbf{u}}
\def\w{\mathbf{w}}
\def\gasenergy{\mathcal{E}}
\def\totenergy{\mathcal{\tilde E}}
\def\Pressure{\mathbb{P}}

\subsection{Two-fluid model.}
Multiplying the Boltzmann equation by powers of velocity and
integrating over velocity space yields fluid equations.
Generic two-fluid equations for a two-species plasma are:
\begin{gather}
 \partial_{t}
    \begin{bmatrix}
      \rho_i \\
      \rho_e \\
      \rho_i\u_i \\
      \rho_e\u_e \\
      \gasenergy_i \\
      \gasenergy_e \\
    \end{bmatrix}
  +\Div
    \begin{bmatrix}
      \rho_i\u_i \\
      \rho_e\u_e \\
      \rho_i\u_i\u_i + \Pressure_i \\
      \rho_e\u_e\u_e + \Pressure_e \\
      \u_i\gasenergy_i+\u_i\cdot\Pressure_i+\q_i\\
      \u_e\gasenergy_e+\u_e\cdot\Pressure_e+\q_e\\
    \end{bmatrix}
 = \begin{bmatrix}
      0 \\
      0 \\
      \qdens_i\E+\J_i\times\B\\
      \qdens_e\E+\J_e\times\B\\
      \J_i\cdot\E\\
      \J_e\cdot\E\\
    \end{bmatrix}
  +
    \begin{bmatrix}
      0 \\
      0 \\
      \R_i\\
      \R_e\\
      \R_i\cdot\u_i + Q_{R,i} + Q_i \\
      \R_e\cdot\u_e + Q_{R,e} + Q_e \\
    \end{bmatrix},
    	\label{eqn:twoFluid}
 \\
 \partial_{t}
    \begin{bmatrix}
      (c\B) \\
      \E
    \end{bmatrix}
 + c 
    \begin{bmatrix}
      \curl\E \\
      -\curl(c\B)
    \end{bmatrix}
 = \begin{bmatrix}
      0 \\
      -\J/{\epsilon_0}
    \end{bmatrix}, \hbox{ and }
  \nabla\cdot
   \begin{bmatrix}
     (c\B) \\
     \E
   \end{bmatrix}
 = 
   \begin{bmatrix}
     0 \\
     \qdens/\epsilon_0
   \end{bmatrix}. \label{eqn:EMfields}
\end{gather}
The variables are defined as follows:
$i$ and $e$ are ion and electron species indices;
for species $s\in\{i,e\}$,
$q_s=\pm e$ is particle charge,
$m_s$ is particle mass,
$n_s$ is particle number density,
$\rho_s = m_s n_s$ is mass density,
$\sigma_s = q_s n_s$ is charge density,
$\J_s = \u_s\sigma_s$ is current density,
$\Pressure_s$ is the pressure tensor,
$\gasenergy_s$ is gas-dynamic energy,
$\q_s$ is the heat flux,
$\R_i=-\R_e$ denotes the interspecies drag force on the ions,
$Q_{R,s}$ denotes heating due to friction (drag), and
$Q_i=-Q_e$ denotes the interspecies thermal heat transfer to the ions.


\subsection{Collisionless isotropic closure.}
\def\idtens{\mathbb{I}}
To close the system we must posit constitutive relations
for the nonevolved quantities.
In a collisionless model we neglect
the terms that come from the collision operator:
$\q_s$, $\R_s$, $Q_{R,s}$, and $Q_s$.
In an isotropic model we assume that the pressure
tensor is a scalar pressure times the identity tensor:
$\Pressure_s=p_s\idtens$; this leads to the constitutive
relation $\gasenergy_s=(3/2)p_s+\rho_s u_s^2/2$.
\footnote{ We remark that our collisionless model with isotropic pressure
seems not to correspond to any general physical regime of plasma
(although it may apply to particular configurations).
A model is considered to be physical if it agrees
with a physical regime in some physical limit.
We assume an isotropic pressure tensor yet no resistivity, but 
for an electron-positron plasma the
time scale over which particles thermalize
is the same as the time scale over which resistive drag force
seeks to equilibrate the velocities of the two species.
}

\subsection{Ohm's law.}
Multiplying the momentum equations of each species by
its charge to mass ratio and summing gives a balance law
for net current.
Invoking the assumption of quasineutrality ($\sigma\approx 0$)
and solving this law for electric field
gives the generalized Ohm's law,
\def\mti{\tilde m_i}
\def\mte{\tilde m_e}
\def\mtr{\tilde m_r}
\begin{gather*}
   \E = \B\times\u + \E',
\end{gather*}
where $\u$ is the mass-averaged fluid velocity and
where the electric field in the frame of
reference of the fluid is the sum of four terms:
\begin{alignat*}{4}
   \E' = &\phantom{+}\eta\cdot\J
           &&\text{ (resistance)}
    \\  &+ \frac{\mti-\mte}{\rho}\J\times\B
           &&\text{ (Hall term)}
  \\ &+ \frac{1}{\rho}\Div( \mte\Pressure_i - \mti\Pressure_e)
           &&\text{ (pressure term)}
    \\ &+ \frac{\mti\mte}{\rho}
         \Big(\partial_{t} \J +\Div\big(
            \u\J + \J\u + \frac{\mte-\mti}{\rho}\J\J\big)\Big)
        && \text{ (inertial term)}.
\end{alignat*}
Here we adopt the convenction that $\tilde m_s:=m_s/e$,
and we have assumed that
the electrical resistance $\R_i\frac{\mti \mte}{\mti+\mte}$
equals $\eta\cdot\J$ (where $\eta$ is the {\em resistivity}),
a simple function of the drift velocity, i.e., of the current.

Ideal MHD assumes that $\E'=0$, and resistive MHD
assumes that $\E'=\eta\cdot\J$.


Substituting Ohm's law into Faraday's law
$\partial_t \B + \curl\E=0$ yields
$\partial_t \B + \Div(\u\B+\B\u) = \curl\E'$,
which implies that the flux of $\B$ through a surface
convected by $\u$ can only change if the curl of
$\E'$ is nonzero.
\footnote{
We remark that if there exists a velocity field
$\v$ for which $\partial_t \B + \curl(\B\times\v)=0$,
then magnetic flux is convected by $\v$ and the topology
of magnetic field lines cannot change.  In particular,
if we merely add the Hall term to the ideal Ohm's law, then
$\partial_t \B + \curl(\B\times(\u+\frac{\mte-\mti}{\rho}\J\times\B))$,
i.e., the magnetic field is essentially carried by the electrons.
Hall-mediated fast reconnection requires a small amount of resistivity
as well.
}

In an electron-positron plasma the masses of ions and electrons
are identical and the Hall term vanishes.
Also, if pressure is isotropic and density varies slowly,
then the curl of the pressure term is zero.
In the collisionless two-fluid model the resistivity is zero.
So for our two-fluid model reconnection could only happen 
by means of the inertial term (or numerical diffusion).

\section{The five-moment two-fluid model}

The collisionless two-fluid equations we solved were
\def\wh{}
\def\pressure{p}
\def\E{\mathbf{E}}
\def\B{\mathbf{B}}
\def\J{\mathbf{J}}
\def\Div{\nabla\cdot}
\def\curl{\nabla\times}
\def\debyeLength{\lambda_D}
\def\gyrofrequency{\omega_g}
\def\mdens{\rho}
\def\u{\mathbf{u}}
\def\gasenergy{\mathcal{E}}
\def\idtens{\mathbb{I}}  
\def\qdens{\sigma}
\def\cc{{\chi}}
\def\diminished{\color{cyan}}
\def\emphasized{\color{blue}}
\begin{gather*}
 \partial_{\wh t}
    \begin{bmatrix}
      \wh\mdens_i \\
      \wh\rho_i\wh\u_i \\
      \wh\gasenergy_i \\
      \wh\mdens_e \\
      \wh\rho_e\wh\u_e \\
      \wh\gasenergy_e \\
    \end{bmatrix}
   +
   \wh \Div
    \begin{bmatrix}
      \wh\rho_i\wh\u_i \\
      \wh\rho_i\wh\u_i\wh\u_i + \wh\pressure_i \, \idtens \\
      \wh\u_i\bigl(\wh\gasenergy_i+\wh\pressure_i\bigr)\\
      \wh\rho_e\wh\u_e \\
      \wh\rho_e\wh\u_e\wh\u_e + \wh\pressure_e \, \idtens \\
      \wh\u_e\bigl(\wh\gasenergy_e+\wh\pressure_e\bigr)\\
    \end{bmatrix}
 = 
    \begin{bmatrix}
      0 \\
      \wh\qdens_i(\wh\E+\wh\u_i\times\wh\B)\\
      \wh\qdens_i\wh\u_i\cdot\wh\E \\
      0 \\
      \wh\qdens_e(\wh\E+\wh\u_e\times\wh\B)\\
      \wh\qdens_e\wh\u_e\cdot\wh\E \\
    \end{bmatrix},
 \\
 \partial_{\wh t}
    \begin{bmatrix}
      \wh\B \\
      \wh\E \\
    \end{bmatrix}
 + \begin{bmatrix}
      \wh\curl\wh\E{\diminished +\cc\nabla\psi} \\
      -c^2\wh\curl \wh\B {\diminished + \cc c^2\nabla\phi} \\
    \end{bmatrix}
 = \begin{bmatrix}
      0 \\
      -\wh\J/{\wh\epsilon} \\
    \end{bmatrix}, 
 {\diminished
 \partial_{\wh t}
    \begin{bmatrix}
      {\diminished\psi} \\
      {\diminished \phi} \\
    \end{bmatrix}
 + 
 }
    \begin{bmatrix}
       {\diminished \cc c^2{\emphasized \wh \Div\wh\B}} \\
       {\diminished \cc{\emphasized \wh\Div\wh\E} } \\
    \end{bmatrix}
 = \begin{bmatrix}
      {\emphasized 0} \\
      {\diminished \wh \cc{\emphasized \wh\qdens/\wh\epsilon}} \\
    \end{bmatrix}.
\end{gather*}
These equations were studied extensively by
Shumlak and Loverich \cite{article:ShLo03}
and Hakim, Shumlak, and Loverich \cite{article:HaLoSh06}.
A version of this model with anisotropic pressure was
also considered by Hakim \cite{article:Hakim07}.

This system is identical in apearance
with the two-fluid system \eqref{eqn:twoFluid}--\eqref{eqn:EMfields}
if the correction potentials $\psi$ and $\phi$,
which we have added for numerical divergence
cleaning purposes, are zero.
These equations imply a wave equation that propagates
the divergence constraint error at the speed $c\chi$.\footnote{
To see this first take the divergence of Maxwell's evolutions
equations.  Then either (1) take the time derivative of Maxwell's
constraint equations to eliminate the electromagnetic field
and get a wave equation for the correction potentials,
or (2) eliminate the correction potentials by taking
the time derivative of the divergence of Maxwell's
evolution equations and the Laplacian of the constraint
equations to eliminate the correction potentials and get
a wave equation for the divergence constraint error.
}
We select $\chi=1.05$.

We nondimensionalized this system by choosing typical
values of magnetic field $B_0$, (ion) number density $n_0$,
particle charge $q_0=e$, and combined particle mass $m_0=m_i+m_e$.
This implies a choice (1) of characteristic time scale
$\omega_g^{-1}:=\frac{m_0}{q_0 B_0}$, where $\omega_g$ is the
cyclotron frequency of a typical particle, (2) of characteristic
velocity $v_A:=\frac{B_0}{\sqrt{\mu_0 m_0 n_0}}$, a typical
Alfv\'en speed, where $\mu_0:=(c^2\epsilon_0)^{-1}$ is the
permeability of magnetic field and (3) straightforwardly
of all other quantities. Replacing every quantity $X$ with
nondimensional representation $\hat X X_0$ gives a system
governing the $\hat X$ values with exactly the same appearance,
except that $1/\hat\epsilon$ = $\hat c^2$.  We drop hats.

%

\section{GEM magnetic reconnection challenge problem}

\def\GEM{\mathrm{GEM}}
With the exception that our nondimensionalized
light speed is 10 rather than their value of 20,
our settings are equivalent to those of \cite{article:Bessho05},
which reflect the settings and conventions of the
original GEM problem \cite{article:GEM}.
To map our SI-like nondimensionalization
onto their Gaussian-like nondimensionalization,
rescale the electromagnetic field by $\B_\GEM = \sqrt{4\pi}\B$
and $\E_\GEM = \frac{\sqrt{4\pi}}{c}\E$.

\subsection{Computational domain.}
The computational domain is the rectangular domain
$[-L_x/2,L_x/2]\times[-L_y/2,L_y/2]$,
where $L_x=8\pi$ and $L_y=4\pi$.
The problem is symmetric under reflection across
either the horizontal or vertical axis.

\subsection{Boundary conditions.}
The domain is periodic in the $x$-axis.
The boundaries perpendicular to the $y$-axis
are thermally insulating conducting wall boundaries.  
A conducting wall boundary is a solid wall boundary
(with slip boundary conditions in the case of ideal plasma)
for the fluid variables, and the electric field at
the boundary has no component parallel to the boundary.
We also assume that magnetic field runs parallel to
and so does not penetrate the boundary (this follows
from Ohm's law of ideal MHD, but we assume it holds
generally).\footnote{We remark that with the symmetries of the
GEM problem at the conducting wall boundary may also be regarded
as a symmetry conditions for a periodic boundary if the solution
on the entire domain is reflected across its bottom boundary
and negated, allowing the GEM problem to be solved on a doubled
domain with periodic boundaries and infinitely smooth initial
conditions.}

\subsection{Model Parameters.}
We carried out simulations for the following choices
of the GEM model parameters:
\begin{align*}
   (1) &&m_i/m_e &= 25, &T_i/T_e&=5  && \hbox{ (original GEM),}
\\ (2) &&m_i/m_e &= 1 , &T_i/T_e&=5, && \hbox{ and}
\\ (3) &&m_i/m_e &= 1 , &T_i/T_e&=1.
\end{align*}

\subsection{Initial conditions.}
The initial conditions are a perturbed Harris sheet equilibrium.
The unperturbed equilibrium is given by
\def\e{\mathbf{e}}
\def\E{\mathbf{E}}
\def\J{\mathbf{J}}
\def\sech{\,\mathrm{sech}}
\begin{align*}
    \B(y) & =B_0\tanh(y/\lambda)\e_x,
  & p(y) &= \frac{B_0^2}{2 n_0} n(y),
 \\ n_i(y) &= n_e(y)
            = n_0(1/5+\sech^2(y/\lambda)),
  & p_e(y) &= \frac{T_e}{T_i+T_e}p(y),
 \\ \E & =0,
  & p_i(y) &= \frac{T_i}{T_i+T_e}p(y).
\end{align*}
On top of this the magnetic field is perturbed by
\begin{align*}
   \delta\B&=-\e_z\times\nabla(\psi), \hbox{ where}
\\ \psi(x,y)&=\psi_0 \cos(2\pi x/L_x) \cos(\pi y/L_y).
\end{align*}
In the GEM problem the initial condition constants are
\begin{align*}
    \lambda&=0.5,
  & B_0&=1,
  & n_0&=1,
  & \psi_0&=B_0/10.
\end{align*}

\section{Properties of the GEM problem}

\subsection{Reconnected flux.}

In defining and discussing magnetic flux we restrict
ourselves to the first quadrant of the domain,
referring to it as if it were the entire domain.

The ideal MHD model implies the frozen-in flux condition, which
says that magnetic flux is convected with the fluid. The boundary
conditions and symmetries of the problem dictate that at the
boundaries fluid can only move parallel to the boundaries and
that the in-plane fluid velocity at the corners must be zero.
Therefore, the frozen-in flux condition would say that the flux
through any boundary must remain constant.

In fact, for models which permit reconnection,
the frozen-in flux condition breaks down near the X-point,
allowing magnetic field to diffuse and allowing field lines
to break and reconnect so that they pass through the horizontal axis.
We therefore define magnetic reconnection to be the loss
of magnetic flux through the vertical axis
into the first quadrant:
\def\leftFlux{F_\mathrm{left}} 
\def\reconFlux{F_\mathrm{recon}} 
\def\max{\mathrm{max}}
\begin{definition}
The reconnected flux $\reconFlux$ is defined by
\begin{align*}
   \leftFlux(t) &:= \int_0^{y_\max} B_1\,dy,
  &\reconFlux(t) &:= \leftFlux(0) - \leftFlux(t).
\end{align*}
\end{definition}

\begin{proposition}
The rate of reconnection is minus the value of the out-of-plane
component of the electric field at the origin (i.e.\ the X-point).
\footnote{
This confirms the theoretical fact that
an MHD model which only includes the $\B\times\u$
and Hall terms in Ohm's law cannot give fast reconnection,
since both these terms must vanish at the origin.
}
\end{proposition}
\begin{proof}
\begin{align*}
   d_t \reconFlux(t) &= - d_t \leftFlux(t)
                     &= - \int_0^{y_\max} \partial_t B_1\,dy
                     &=   \int_0^{y_\max} \partial_y E_3\,dy
                     &= - E_3(0),
\end{align*}
since $E_3$ is zero at the conducting wall.
\end{proof}


\subsection{Reflectional symmetries.}

The GEM problem has reflectional symmetry across the
horizontal and vertical axes.  We impose this symmetry
by restricting our computations to the first quadrant.
\footnote{
In simulations which incorporate the entire domain
symmetry is often lost due to computational noise and
the inherent instability of the problem.
}
At the origin, symmetries across the horizontal and vertical
axes mean that vectors have only an out-of-plane component
and pseudovectors must be zero.

\subsection{Reductions in Ohm's law at the origin.}

Symmetries at the origin reduce Ohm's law to:
\begin{alignat*}{3}
   \E_3 = &\eta J_3
     + \frac{1}{\rho}
        (\mte(\partial_{x_1}\Pressure_{i,1,3}+\partial_{x_2}\Pressure_{i,2,3})
       - \mti(\partial_{x_1}\Pressure_{e,1,3}+\partial_{x_2}\Pressure_{e,2,3}))
 \\            &+ \frac{\mti\mte}{\rho} \Big(\partial_{t} \J_3  
           + J_3\Div\u + u_3\Div\J + \frac{\mte-\mti}{\rho} J_3\Div\J\Big)
\end{alignat*}
If we neglect resistivity, off-diagonal pressure components,
and assume that the flow of charge and mass toward or away from the
origin is small, then this becomes
\begin{gather*}
   -d_t \reconFlux = \E_3 \approx \frac{\mti\mte}{\rho} \partial_t \J_3 \approx \mti\mte\partial_t(\J_3/\rho),
\end{gather*}
i.e., the value of $-\mti\mte \J_3/\rho$ at the origin will tend to track
with $\reconFlux$. 

%

\section{Results}

Following the precedent of \cite{article:ShLo03},
\cite{article:Hakim07}, and
\cite{article:HaLoSh06},
we carried out simulations of the GEM problem 
using a third-order shock-capturing
Runge-Kutta Discontinuous Galerkin solver for
a collisionless two-fluid model with isotropic pressure
for each species.  (We enforced the divergence cleaning
for the magnetic field but not for the electric field.)
We carried out simulations on a quarter domain
(hence enforcing symmetry) for mesh sizes of
$32\times 16$,
$64\times 32$, and
$128\times 64$.

For the low mesh resolution we seemed to observe fast reconnection
in the electron-positron plasma (based on the pattern of
magnetic field lines and the rates of reconnection), but not for high resolutions.
We hypothesize that the electron inertial term coupled with
sufficient (numerical) resistivity is sufficient to yield
fast reconnection.  In future studies we hope to explicitly introduce
collisional (diffusive) terms (rather than relying on numerical
diffusion) and explore whether we can show convergence to
fast reconnection for resistive isotropic two-fluid plasma.

We plotted reconnected flux, $-\mti\mte J_3/\rho$, and a crude
calculation of -$\int_0^t E_3$ (the cumulative sum of $-E_3$
evaluated at integer times) for all three combinations of mass
ratio and temperature ratio.
Our plots indicate that
$-\mti\mte J_3/\rho$ only tracked with reconnected flux for the
high-resolution simulations of electron-positron plasma up until
the time when a large magnetic island formed around the origin.
(These were the plots where fast reconnection did not occur.)

\begin{figure}[t!]
  \includegraphics[width=150mm, height=60mm]{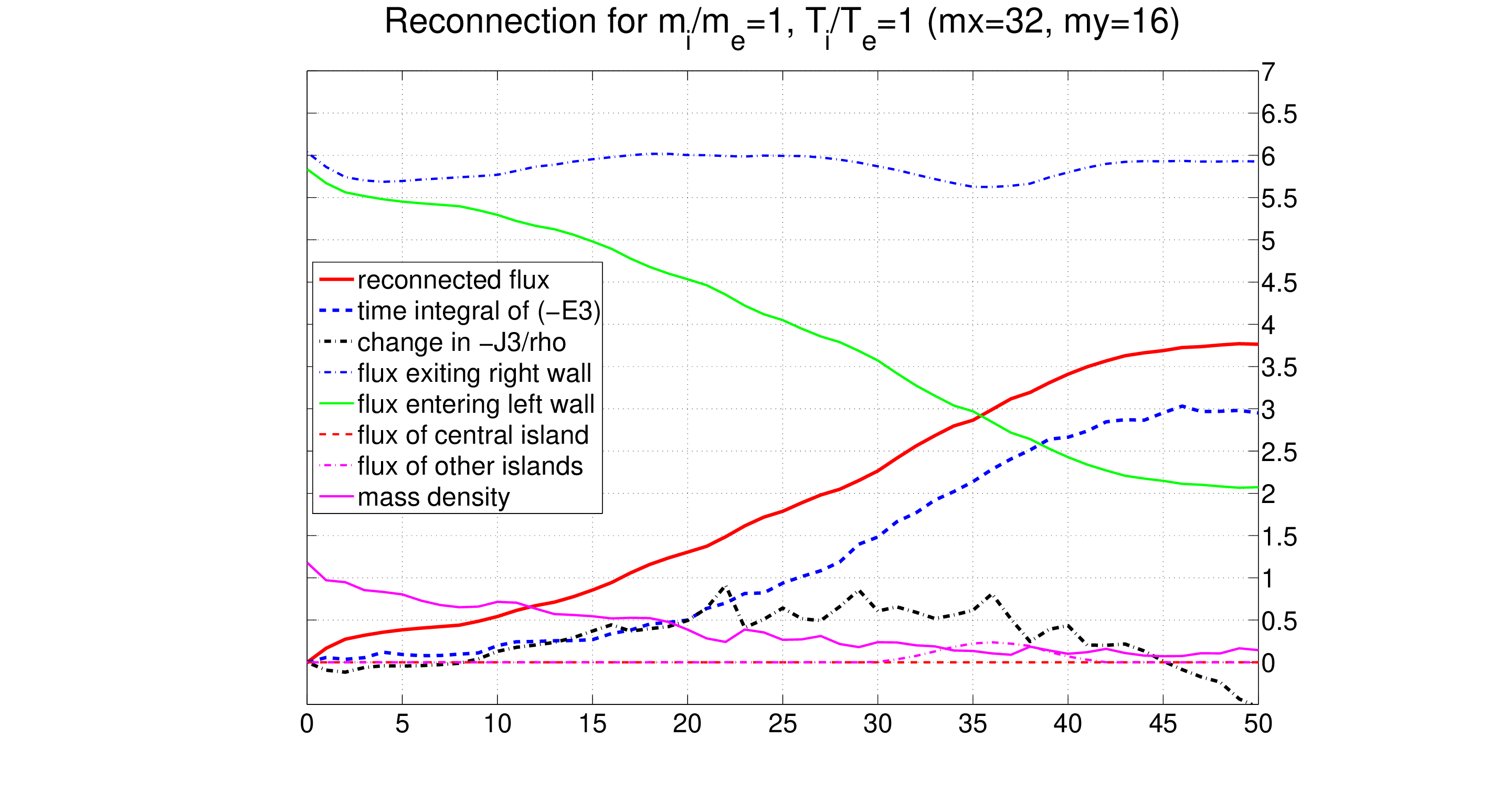}
  \\
  \includegraphics[width=150mm, height=60mm]{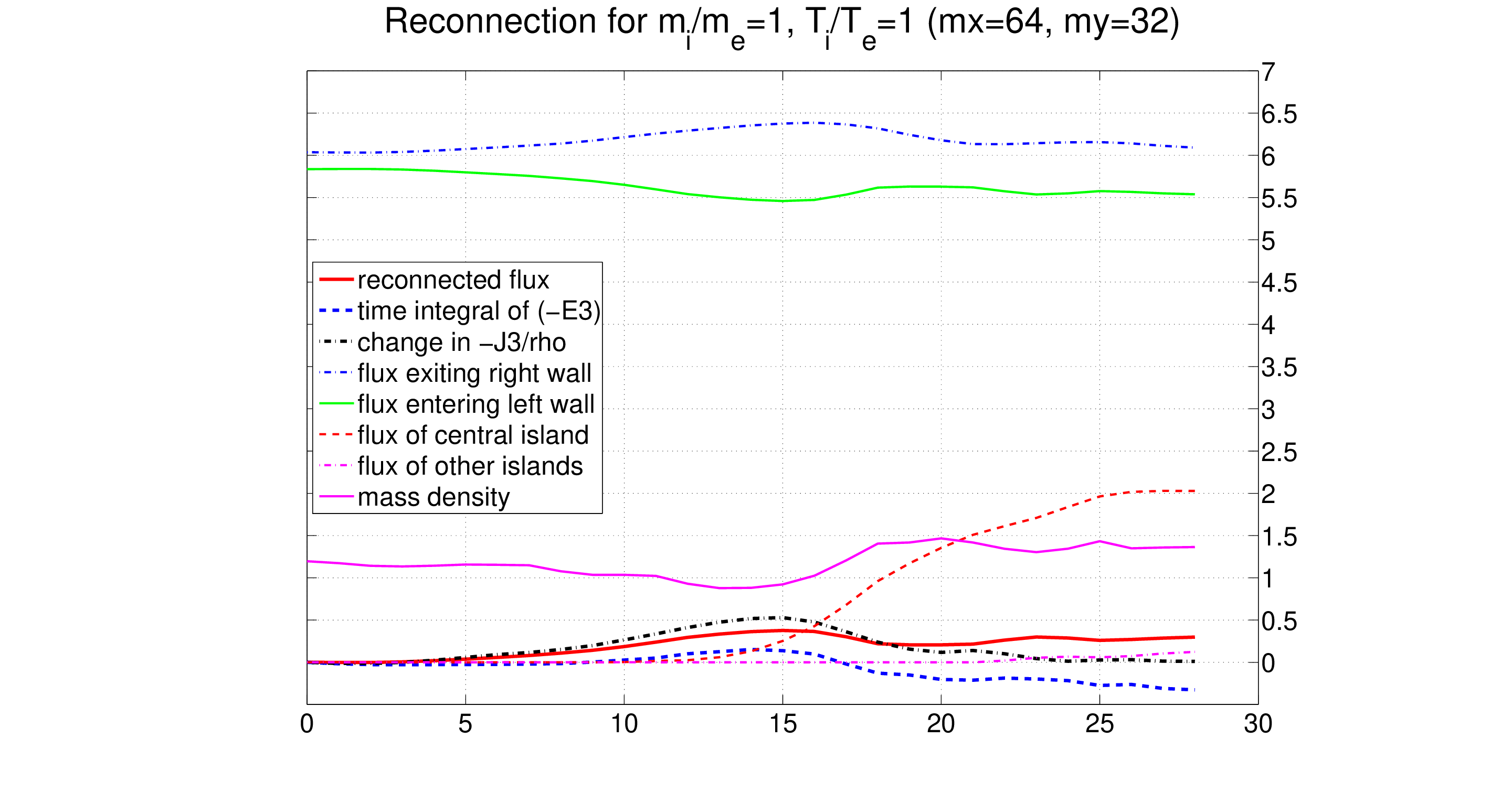}
  \\
  \includegraphics[width=150mm, height=60mm]{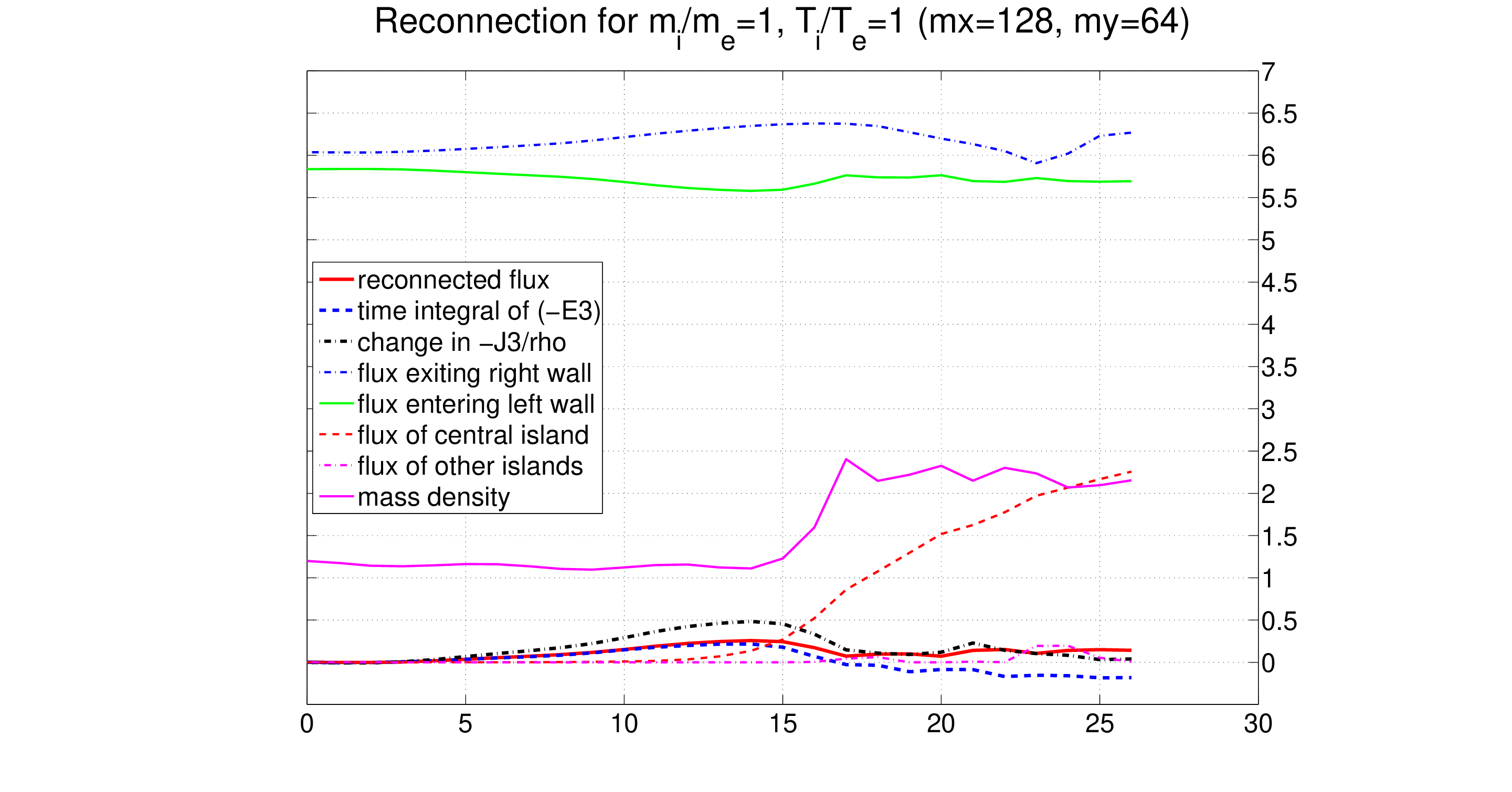}
  \caption{\small Three plots attempting to show convergence 
    of reconnection data for electron-positron plasma with
    equal initial temperatures.  Reconnection is suppressed for 
    finer mesh resolution but fast reconnection occurs for coarse
    mesh resolution.  We conjecture that numerical diffusion
    coupled with the ion inertial term is sufficient to admit
    fast reconnection.  The appearance of a large central magnetic
    island beginning around time 15 is curious.
    }
  \label{f:plot_1_1}
\end{figure}
\begin{figure}[t!]
  \includegraphics[width=150mm, height=60mm]{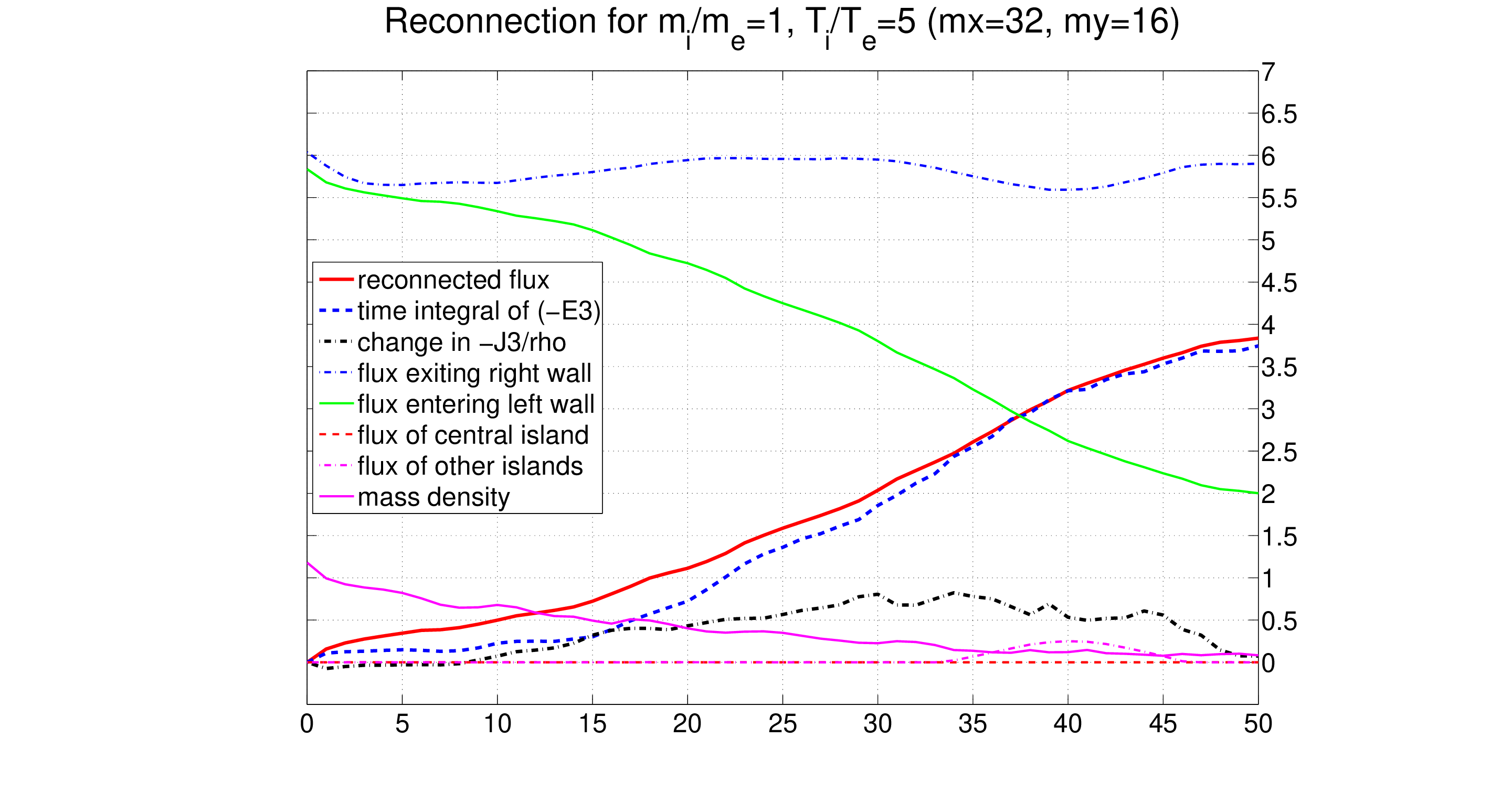}
  \\
  \includegraphics[width=150mm, height=60mm]{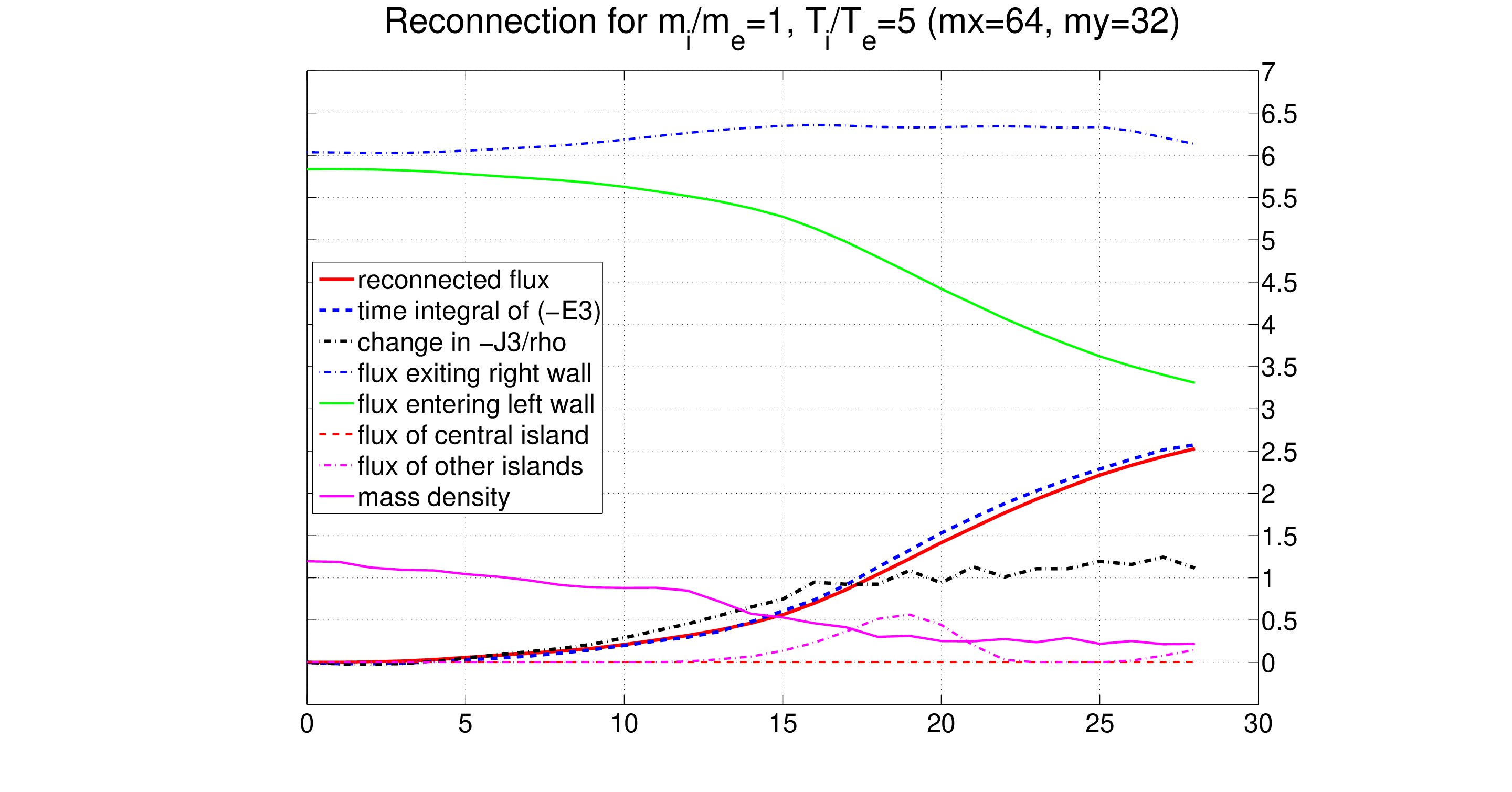}
  \\
  \includegraphics[width=150mm, height=60mm]{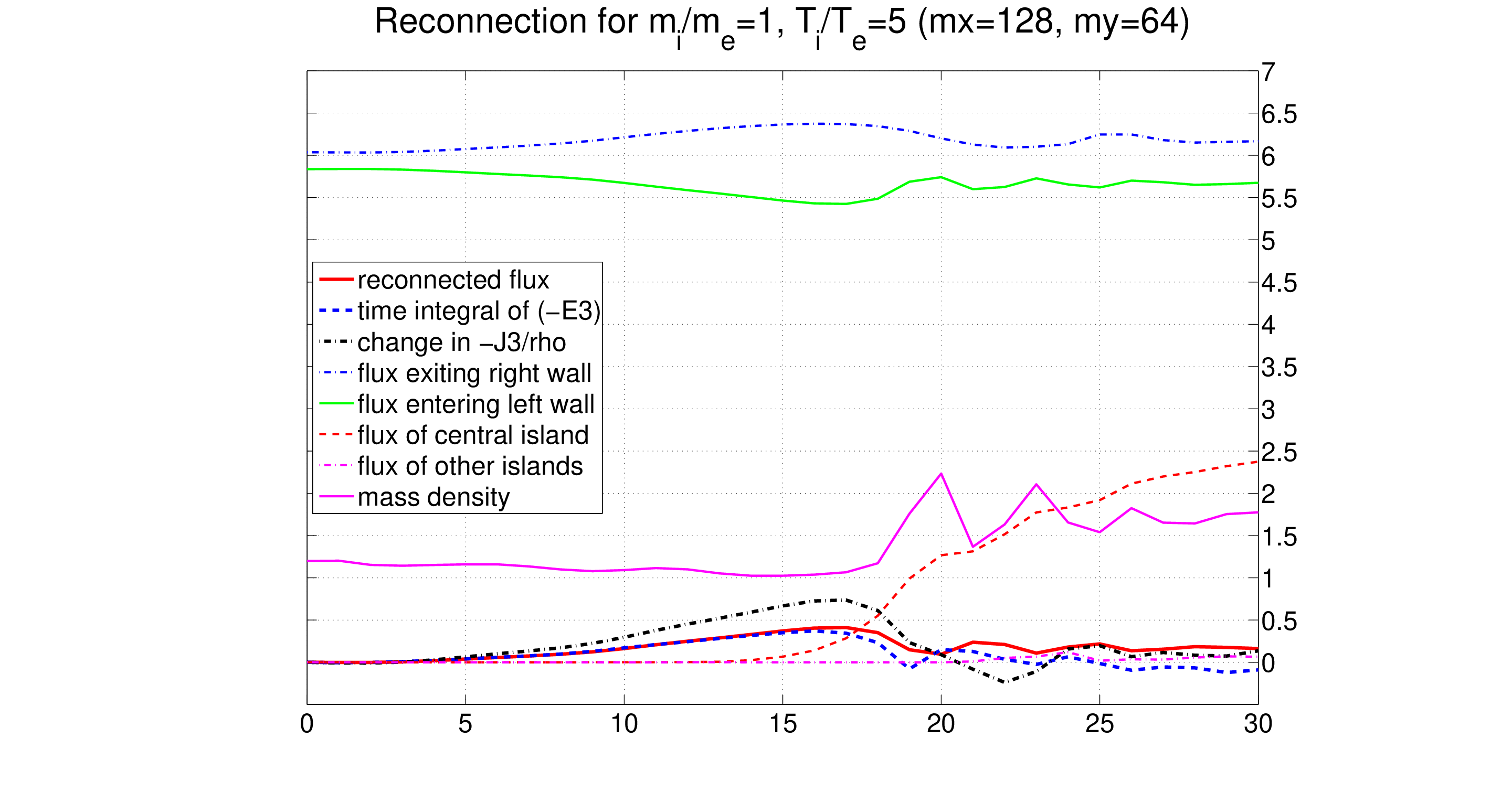}
  \caption{\small Convergence study
    of reconnection for electron-positron plasma with initial
    temperature ratio of 5.  The reconnection rate would appear to
    decrease with decreasing numerical diffusion.  Only for the
    finest mesh does a central magnetic island form.
    }
  \label{f:plot_1_5}
\end{figure}
\begin{figure}[t!]
  \includegraphics[width=150mm, height=60mm]{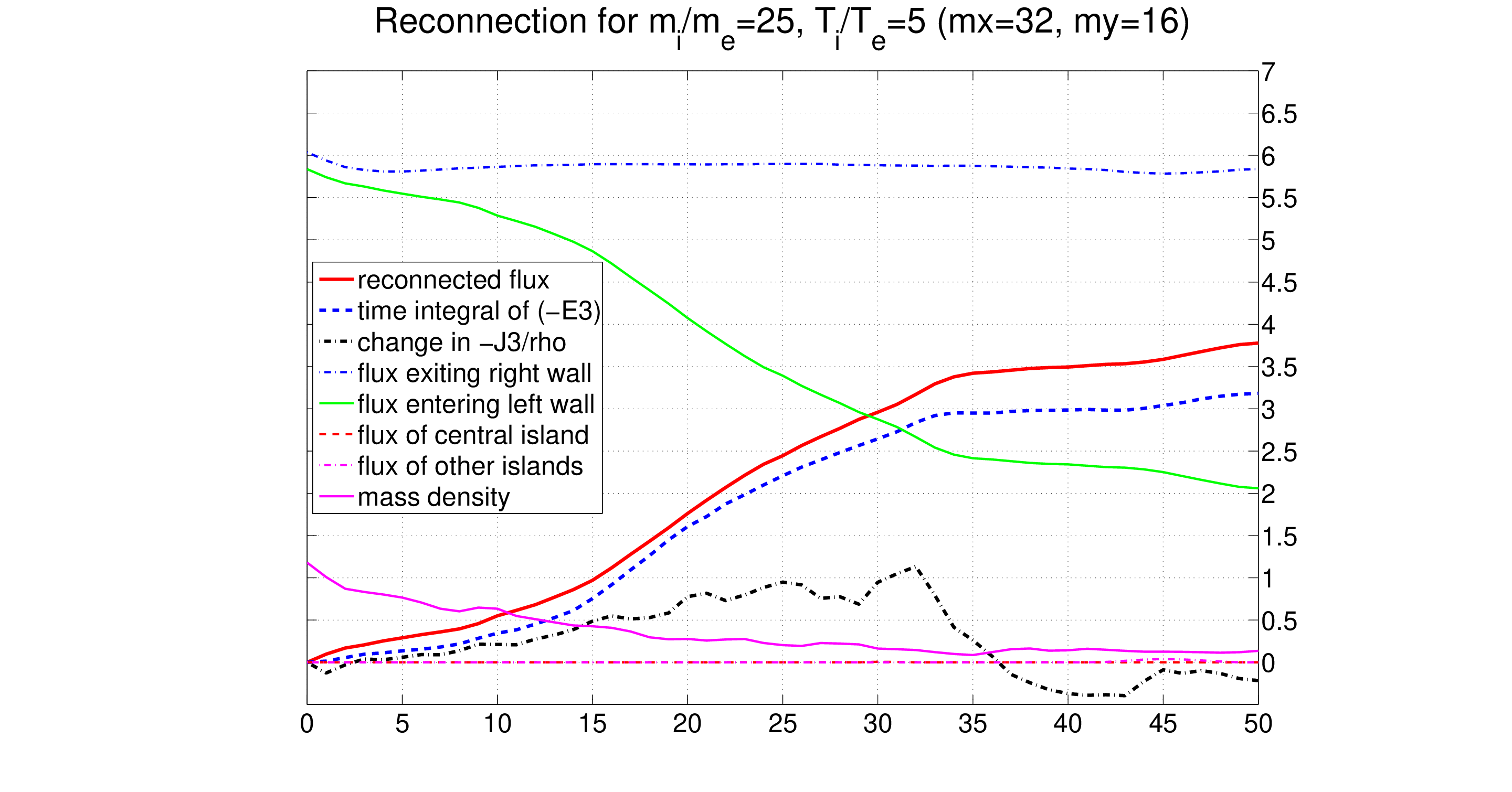}
  \\
  \includegraphics[width=150mm, height=60mm]{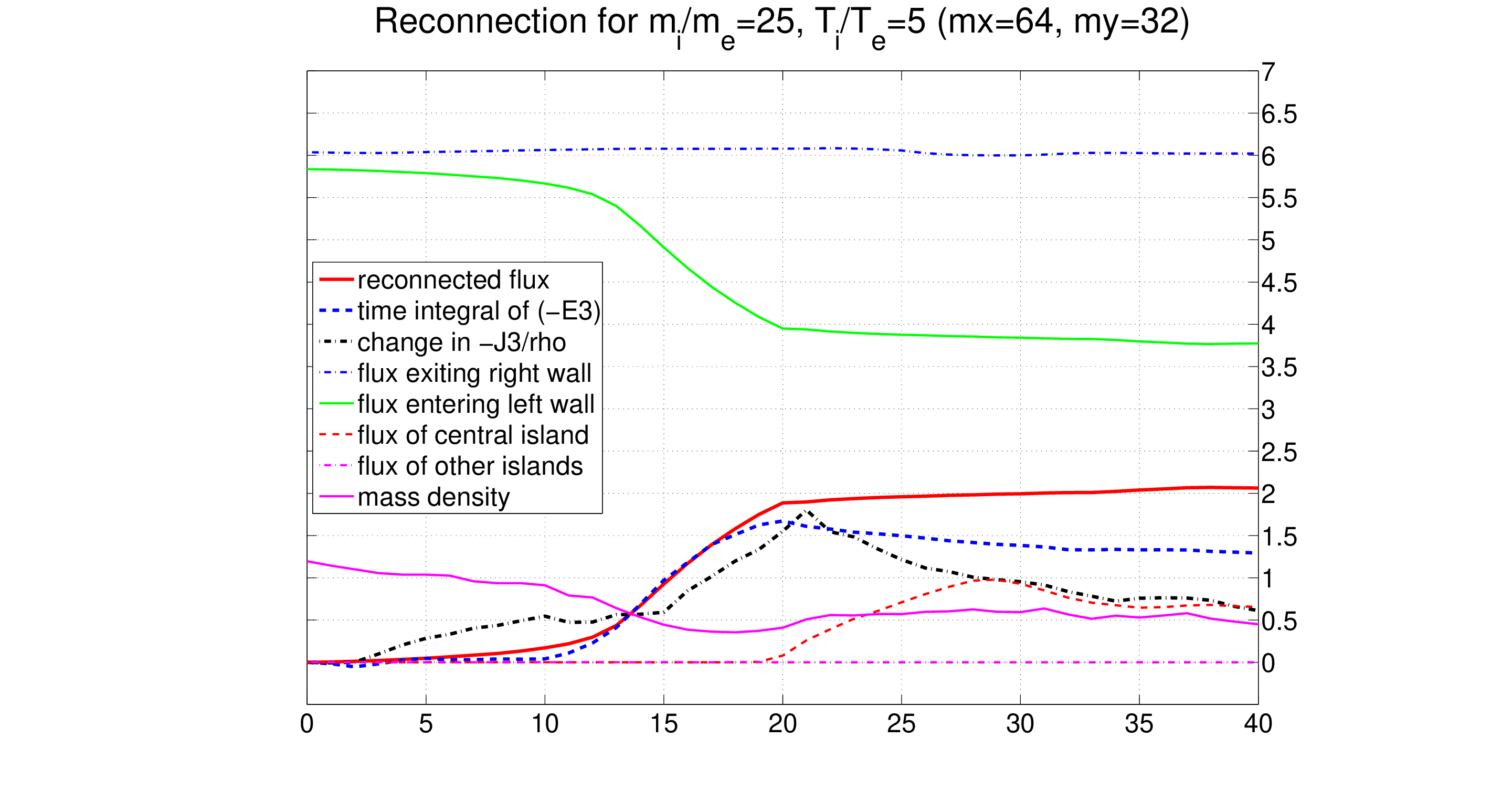}
  \\
  \includegraphics[width=150mm, height=60mm]{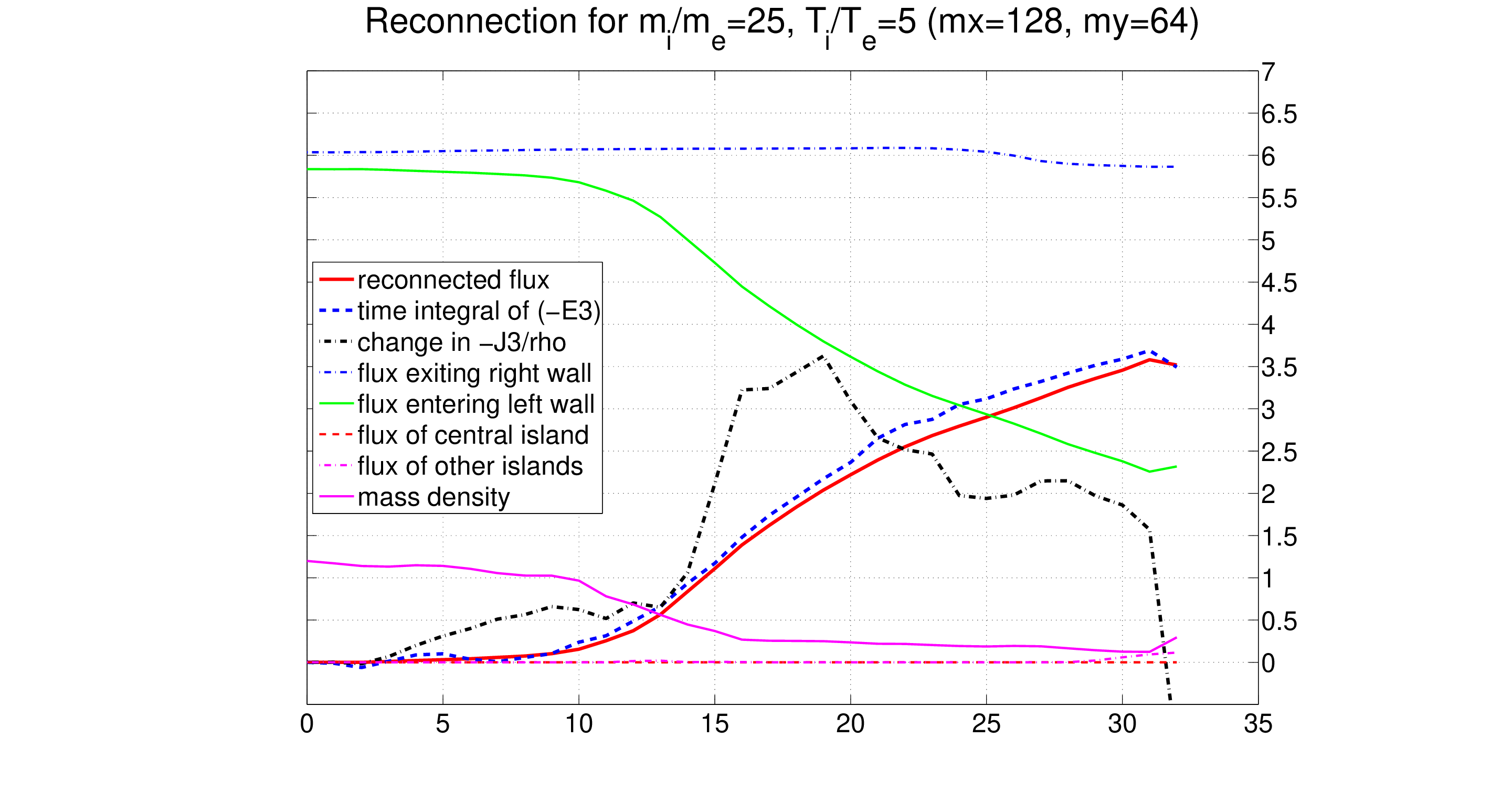}
  \caption{\small Convergence study
    of reconnection for electron-positron plasma with mass ratio of 25
    and initial temperature ratio of 5.  Fast reconnection occurs for
    both coarse and fine mesh resolutions.  The anomalously early 
    cessation of reconnection in the simulation at
    intermediate resolution is evidently related to the appearance of
    a large central magnetic island beginning around time 20.
    }
  \label{f:plot_25_5}
\end{figure}


\end{document}